\documentclass[10pt]{article}
\usepackage{amssymb}
\usepackage{amsmath}

\usepackage{amsthm}
 
\theoremstyle{definition}
\newtheorem*{definition}{Definition}
\newtheorem*{definitions}{Definitions}
\newtheorem{theorem}{Theorem}
\newtheorem{lemma}{Lemma}
\newtheorem{remark}{Remark}

\begin{document}\title{A Simple Bijection Between 231-Avoiding and 312-Avoiding Placements}
\author{Jonathan Bloom\\Dartmouth College \and Dan Saracino\\Colgate University}
\date{}
\maketitle
\begin{abstract}
Stankova and West proved in 2002 that the patterns 231 and 312 are shape-Wilf-equivalent.  Their proof was nonbijective.  We give a new characterization of 231 and 312 avoiding full rook placements and use this to give a simple bijection that demonstrates the shape-Wilf-equivalence. 
\end{abstract}

\section{Introduction}

For any pattern $\tau\in S_k$, let $S_n(\tau)$ denote the set of permutations in $S_n$ that avoid $\tau$, in the sense that they have no subsequence order-isomorphic to $\tau$. For any Ferrers board $F$, let $S_F(\tau)$ denote the set of all full rook placements on $F$ that avoid $\tau$.  We say  that two patterns $\tau$ and $\sigma$ are \emph{Wilf-equivalent}, and write $\tau \sim \sigma$, if $|S_n(\tau)|=|S_n({\sigma})|$ for all $n> 0$.  We say that $\tau$ and $\sigma$ are \emph{shape-Wilf-equivalent}, and write $\tau \sim_s \sigma$,  if $|S_F(\tau)|=|S_F(\sigma)|$ for all $F$.  So shape-Wilf-equivalence implies Wilf-equivalence, as we see by considering square Ferrers boards. (The relevant definitions will be reviewed in Section 2.)

The concept of shape-Wilf-equivalence was introduced in [1], as a means for obtaining results about Wilf-equivalence. Since shape-Wilf-equivalence is stronger than Wilf-equivlence,  positive results about it are rare.  The only ``general result" was obtained in [1], where it was shown that the patterns $k\cdots 321$ and $123\cdots k$ are shape-Wilf-equivalent for every positive $k$. Later, in [4], Stankova and West proved that the patterns 231 and 312 are shape-Wilf-equivalent, and the motivation for our paper comes from their result.  Their proof that $|S_F(231)|=|S_F(312)|$ was nonbijective, and somewhat complicated.  Our purpose here it to give a simple bijection between $S_F(231)$ and $S_F(312)$.

We will do so by associating a sequence of nonnegative integers to each full rook placement on $F$, and characterizing those sequences that arise from 231-avoiding or 312-avoiding placements. We will give a simple way to transform a sequence arising from a 231-avoiding placement into a sequence arising from a 312-avoiding placement, and vice-versa.

In Section 2, we will review the needed definitions, define our bijection, and state the Theorems needed to verify that it is indeed a bijection.  In Sections 3 and 4 we will prove these theorems.

Vit Jel\'inek has pointed out to us that a bijective proof of the shape-Wilf-equivalence of the patterns 231 and 312 can also be obtained from his work on pattern--avoidance in matchings.  See [3], where a bijection is obtained by establishing an isomorphism between generating trees.  Examples on small Ferrers boards show that Jel\'inek's bijection differs from ours.  

\section{The bijection}

\begin{definitions}
Let $\mathcal{A}$ be an $n\times n$ array of unit squares and coordinatize it by placing the bottom left corner of $\mathcal{A}$ at the origin in the $xy$-plane.   We refer to the corners of all the squares in $\mathcal{A}$ as \emph{vertices} and reference them by their cardinal position.  For example, the upper right corner will be called the NE corner.   For any vertex $V=(a,b)$ we define $R(V)$ to be the rectangular array of squares bounded by the lines $x=0, x=a, y=0, $ and $y=b$.  

A \emph{Ferrers board} is any subset $F$ of $\mathcal{A}$ with the property that $R(V)\subseteq F$ for each vertex in $F$.  We define the \emph{right/up border} of $F$ to be the border of $F$ excluding the vertical left hand side and horizontal bottom.
\end{definitions}

Next we need to define the generalization of a permutation for the context of Ferrers boards. 

\begin{definitions}
A \emph{rook placement} on a Ferrers board $F$ is a subset of $F$ that contains at most one square from each column of $F$ and at most one square from each row of $F$.  We indicate these squares by putting markers in them.  Likewise a \emph{full rook placement} is a rook placement such that each row and each column has exactly one marker in it.  We say a rook placement $P$ on a Ferrers board $F$ \emph{avoids} $\tau$ if and only if for every vertex $V$ on the right/up border the permutation that is order-isomorphic to the restriction of $P$ to $R(V)$ avoids $\tau$ in the usual sense.  
\end{definitions}

\begin{definition}
For any rook placement $P$ on $F$ and any vertex $V$ of $F$, we denote by $S(P,V)$ the maximal length of an increasing sequence of $P$ in $R(V)$.
\end{definition}

To define our bijection from $S_F(231)$ to $S_F(312)$, we first associate to each full rook placement $P$ on $F$ a sequence $S(P,F)$.

\medskip

\noindent \textbf{Notation}.  For any full rook placement on a Ferrers board $F$, $S(P,F)$ denotes the sequence of nonnegative integers obtained by taking $S(P,V)$ for all $V$ on the right/up border of $F$, starting with the vertex at the top left corner of $F$.

\begin{theorem}\label{thm1} 
If $P$ is in $S_F(231)$ or $S_F(312)$, then $S(P,F)$ and $F$ determine $P$.  
\end{theorem}

We will prove Theorem~\ref{thm1} in Section 3 by giving a ``reverse algorithm" for the map $P \rightarrow S(P,F)$ consequently establishing injectivity.
  
Readers familiar with Fomin's growth diagram algorithm will note that the values of $S(P,F)$ are the first entries of the partitions in the oscillating tableaux produced by the algorithm.  Theorem~\ref{thm1} may be restated by saying that for $P$ in $S_F(231)$ or $P$ in $S_F(312)$ the first entries in the partitions determine the oscillating tableaux.  

To define our bijection, we will need to characterize those sequences that arise from $P\in S_F(231)$ or $P\in S_F(312)$.

\begin{definition}  
If $F$ is a Ferrers board, then an \emph{F-sequence} is a sequence of nonnegative integers assigned to the vertices on the right/up border of $F$, starting with the vertex at the top left corner.
\end{definition}

\begin{definition}[the 231-conditions]
If $F$ is a Ferrers board and $S$ is an $F$-sequence, then the \emph{231-conditions for the pair $(F,S)$} are the following three conditions:

\begin{itemize}
\item[(i)]  (monotonicity conditions)  If $V_1$ and $V_2$ are vertices on the right/up border and $V_1$ is either directly to the left of $V_2$ or directly below $V_2$ then $S(V_1)\leq S(V_2)\leq S(V_1)+1$.
\item[(ii)] (0-conditions)  The first and last values of $S$ are 0, and there do not exist consecutive vertices $V_1$ and $V_2$ such that $S(V_1)=0=S(V_2)$.
\item[(iii)] (diagonal condition) If $V_1$ and $V_2$ are vertices on the right/up border that are at the left and right ends of a diagonal with slope $-1$ that lies entirely within $F$, then $S(V_1)\leq S(V_2)$.

\end{itemize}
\end{definition}

\begin{definition}[the 312-conditions]
With $S$ as in the preceding definition, the \emph{312-conditions for the pair $(F,S)$} are the same as the 231-conditions, except that we reverse the inequalities in the diagonal condition.
\end{definition}

The following definition is often useful when dealing with the diagonal condition.

\begin{definition}
We refer to a pair of vertices $V_1, V_2$ as \emph{diagonal vertices} or \emph{F-diagonal vertices} if they are on the right/up border of $F$ and are at the left and right ends of a diagonal with slope $-1$ that lies entirely within $F$.
\end{definition}

\begin{theorem}\label{thm2}
If $F$ is a Ferrers board whose longest row and longest column have the same length, and $S$ is an $F$-sequence, then there exists $P\in S_F(231)$ (respectively, $P\in S_F(312)$) such that $S(P,F)=S$ if and only if $(F,S)$ satisfies the 231-conditions (respectively, the 312-conditions).
\end{theorem}

Theorem~\ref{thm2} will be proved in Section 4.

\medskip

To obtain our bijection,  we need  a way to take $P\in S_F(231)$ (respectively, $P\in S_F(312)$ and transform $S(P,F)$ into a sequence satisfying the 312-conditions (respectively, the 231-conditions).  To do this we need our first  lemma.

\begin{lemma}\label{vertex_to_number_lemma}
For any Ferrers board $F$ and vertex $V$ on its right/up border, there exists an integer $N(F,V)$ such that for every full rook placement $P$ on $F$, there are exactly $N(F,V)$ markers of $P$ in $R(V)$.
\end{lemma}

\begin{proof}
Take any full rook placement $P$ on $F$. We proceed inductively, starting with the vertex $V$ at the top left corner. Clearly, $P$ has no markers in $R(V)$.  If $V_1, V_2$ are vertices on the right/up border such that $V_1$ is either directly to the left of $V_2$ or directly below it, then the number of markers of $P$ in $R(V_2)$ is one greater than the number in $R(V_1)$.   \end{proof}

\begin{definition}
If $P$ is a full rook placement on a Ferrers board $F$, and $S=S(P,F)$, then we define another $F$-sequence $S^+$ by letting $S^+(V)=0$ if $S(V)=0$, and $S^+(V)=N(F,V)+1-S(V)$ otherwise.
\end{definition}

It is clear that $S^+$ is an $F$-sequence, because $S(V)\leq N(F,V)$.

\medskip

\begin{lemma} \label{bijection_lemma}
Let $P$ be a full rook placement on $F$ and let $S=S(P,F)$. Then if $(F,S)$ satisfies the 231-conditions (respectively,  the 312-conditions), $(F,S^+)$ satisfies the 312-conditions (respectively, the 231-conditions).
\end{lemma}

\begin{proof}
We give the proof when $(F,S)$ satisfies the 231-conditions. The proof of the other case is nearly identical.

To verify the monotonicity conditions for $(F,S^+)$, first let $V_1, V_2$ be vertices on the right/up border of $F$ such that $V_1$ is directly to the left of $V_2$.  In the case that $S(V_1)=0$ then $S(V_2) = 1$ by the $231$-conditions.  Observe that $N(F,V_2) = 1$ as well, which implies that $S^+(V_1) = 0$ and $S^+(V_2)=1$.  In the case that $S(V_1)\neq 0$ then we know that $S(V_1) \leq S(V_2) \leq S(V_1) +1$.  Since $N(F,V_1) +1 = N(F,V_2)$ then we get
$$ N(F,V_1) +2 - S(V_1) \geq N(F,V_2)+1 - S(V_2) \geq N(F,V_1) +1 - S(V_1)$$
and hence  $1+S^+(V_1) \geq S^+(V_2) \geq S^+(V_1)$. The proof is the same if $V_1$ is directly below $V_2$ and therefore monotonicity holds.  

The 0-conditions hold for $(F,S^+)$ because $S^+(V)=0$ if and only if $S(V)=0$.

To verify the 312-diagonal condition for $(F,S^+)$, let $V_1, V_2$ be $F$-diagonal vertices.  We note that $N(F,V_1)=N(F,V_2),$  because $N(F,V)$ increases by one each time we move to the right on the right/up border, and decreases by one each time we move downward, and the number of rightward steps between $V_1$ and $V_2$ equals the number of downward steps.  By the 231-diagonal condition for $(F,S)$, we have $S(V_1)\leq S(V_2)$. If  $S(V_1) \neq 0$, then since $N(F,V_1)=N(F,V_2)$, 
we have $S^+(V_1)\geq S^+(V_2)$.  If $S(V_1) = 0$ then  $N(F,V_1)=0$ so $N(F,V_2)=0$ and thus $S(V_2)=0$.
\end{proof}

\begin{definitions}
Let $P\in S_F(231)$ and let $S=S(P,F)$.  By Theorems 1 and 2, let $\alpha(P)$ denote the unique element of $S_F(312)$ such that $S(\alpha(P),F)=S^+$.  For $P\in S_F(312)$, define $\beta(P)\in S_F(231)$ analogously.
\end{definitions}

\begin{theorem}\label{thm4}
The maps $\alpha: S_F(231)\rightarrow S_F(312)$ and $\beta:S_F(312)\rightarrow S_F(231)$ are inverses, and therefore both are bijections.
\end{theorem}

\begin{proof}
This follows from the fact that if $S=S(P,F)$ for $P$ in either $S_F(231)$ or $S_F(312)$, then $S^{++}=S$.
\end{proof}

\begin{remark}
Although our proofs depend on the fact that we are working with full rook placements it follows from Theorem 3 that for any Ferrers board $F$ the number of 231-avoiding rook placements on $F$ is equal to the number of 312-avoiding rook placements on $F$.  The idea is as follows.  For any rook placement $P$ on $F$ we have the the set $\mathcal{C}$ of column numbers corresponding to columns that contain a marker.    Similarly we get the set $\mathcal{R}$ of row numbers.  Now we may consider the set of squares
$$F_P =\{ (c,r) | c\in \mathcal{C}, r\in \mathcal{R}\}.$$   
Observe that we may view $F_P$ as a Ferrers board by sliding all the squares down and then left.  Likewise, $P$ may be viewed as a full rook placement on $F_P$.   We may now define an equivalence relation $\sim$ on rook placements by saying two placements $P$ and $Q$ are related if and only if $F_P = F_Q$.  Now let $A$ (respectively, $B$) be the partition under $\sim$ of the set of 231-avoiding (respectively, 312-avoiding) rook placements on $F$. Clearly $|A| = |B|$ and if $\overline{P}\in A$ then Theorem 3 implies that $|\overline{P}| = |\overline{\alpha(P)}|$ proving our claim.  
\end{remark}

\section{The reverse algorithm}

We will prove Theorem 1 by developing an ``reverse algorithm" for the map $P \rightarrow S(P,F)$. To do this, we must first establish some properties of $S(P,V)$.

\begin{lemma}\label{inc_by_1} 
Let $P$ be a  rook placement on  Ferrers board $F$, and let $V_1$ and $V_2$ be vertices of $F$. Then if $V_1$ is directly to the left of $V_2$, or directly below $V_2$, we have
$$S(P,V_1)\leq S(P,V_2)\leq S(P,V_1)+1.$$
\end{lemma}

\begin{proof}
This follows immediately from the definition of $S(P,V)$.
\end{proof}

\begin{lemma}\label{lemma_growthrule}
Suppose $P$ is a rook placement on a Ferrers board $F$, and $A,B,C$ are the vertices at the $NW,NE$, and $SE$ corners, respectively, of a square $\mathcal{B}$ in $F$. Let $a,b,c$ be the values of $S(P,V)$ at $V=A,B,C$, respectively.   Then if $P$ has no marker in $\mathcal{B}$, we have $b=\textrm{max}(a,c)$. And $P$ has a marker in $\mathcal{B}$ if and only if $b=a+1=c+1$.
\end{lemma}

\begin{proof}  
First suppose $P$ has no marker in $\mathcal{B}$. Consider an increasing sequence $I$  of length $b$ in $R(B)$.   If $I$ is contained in $R(C)$, then $b\leq c$.  If $I$ is not contained in $R(C)$, then $I$ must include a marker in the top row of $R(B)$, so $I$ terminates at this marker, which is to the left of $\mathcal{B}$, and therefore $I$ is contained in $R(A)$, yielding $b\leq a$. In either case, $b\leq \textrm{max}(a,c)$.  Since the reverse inequality follows from Lemma \ref{inc_by_1}, we have $b=\textrm{max}(a,c)$.
  
  It follows that if $P$ has no marker in $\mathcal{B}$ then we cannot have $b=a+1=c+1$.  It is clear that if $P$ has a marker in $\mathcal{B}$ then $b=a+1=c+1$.
\end{proof}

\begin{lemma}\label{equality_of_labels}
Suppose $P\in S_F(231)$ and $V_1,V_2$ are vertices  of $F$ such that $V_1$ is directly below $V_2$. Suppose $P$ has a marker $X$  in the top row of $R(V_2)$, and another marker $Y$ in $R(V_2)$  that is to the right of $X$.  Then $S(P,V_1)=S(P,V_2)$.
\end{lemma}

\begin{proof}
Since $P\in S_F(231)$, $P$ has no 231-patterns in $R(V_2)$.  If $R$ is the set of markers of $P$ in $R(V_2)$ that are to the right of $X$, and $L$ is the set of markers of $P$ in $R(V_2)$ that are to the left of $X$, it follows that all elements of $R$ are in higher rows than all elements of $L$.  Since $R\neq \emptyset$ because of the presence of $Y$,  both $S(P,V_1)$ and $S(P,V_2)$ are the sum of the maximal length of an increasing sequence in $L$ and the maximal length of an increasing sequence in $R$.  This proves the lemma.  
\end{proof}

\begin{proof}[Proof of Theorem 1]
It will suffice to prove the result for $P\in S_F(231)$, for then  by considering the inverse placement $P'$ on the conjugate board $F'$, we obtain the result for $P\in S_F(312)$. 

So let $P\in S_F(231)$. Suppose the bottom row of $F$ contains exactly $n$ squares, and the right-hand column of $F$ contains exactly $r$ squares.   Let the values of $S(P,V)$ on the line $x=n$ be $b_r,
\ldots, b_0$, from top to bottom, and let the values on the line $x=n-1$ be $a_r,\ldots, a_0$, again from top to bottom. The values $a_r,b_r,b_{r-1},\ldots, b_0$ are included in $S(P,F)$, and we will show that from these values we can determine the location of the marker of $P$ in the right-hand column, and the values of $a_{r-1},\ldots, a_0$.

Choose $j$ as large as possible such that $b_j > b_{j-1}$. Then the markers $X_r,\ldots,\\ X_{j+1}$ of $P$ in rows $r, \ldots, j+1$ are not in the right-hand column, and since there is a marker $Y$ in the right-hand column and there are no 231-patterns in 
$R((n,r))$, the markers $X_r,\ldots, X_{j+1}$ must form a decreasing sequence. Applying 
Lemma \ref{equality_of_labels} repeatedly, with the $X$ of that lemma being $X_r,\ldots, X_{j+1}$ in turn, we conclude that $a_r=\cdots =a_{j+1}$.  The marker $X_j$ in row $j$ must be $Y$, for else, using $X_j,Y$ and Lemma \ref{equality_of_labels}, we would have $b_j=b_{j-1}$. It follows that $a_j=b_{j-1}$ and $a_i=b_i$ for $i\leq j-1$.

We have determined the placement of the marker $Y$ in the right-hand column and the values of  $a_{r-1},\ldots, a_0$. If we delete the right-hand column and the row containing the marker $Y$ we obtain a smaller board $F^*$ and a placement $P^*\in S_{F^*}(231)$ such that the sequence of values $S(P^*,F^*)$ is $S(P,F)$ with the terminal  $r+1$ values $b_r, \ldots, b_0$ replaced by the $r-1$ values $a_{r-1},
\ldots, a_j,a_{j-2}, \ldots, a_0$. By iterating the above argument we can proceed to determine the  positions of all the markers in $P$, from right to left.
\end{proof}

\section{The proof of Theorem 2}

To prove Theorem 2, it will suffice to prove the assertion about $P\in S_F(231)$, for the assertion about $P\in S_F(312)$ then follows by considering the inverse placement $P'$ on the conjugate board $F'$, with $P' \in S_{F'}(231)$.

\medskip

\noindent \textbf{Notation}. Let $F$ be a Ferrers board whose longest row and longest column each contain exactly $n$ squares.  Let $\mathcal{B}$ be the square at the top of the right-hand column of $F$, and suppose $\mathcal{B}$ is in row $r$. Let $A,B,C$ be the vertices at the $NW, NE,$ and $SE$ corners of $\mathcal{B}$.

\medskip

We will first prove the necessity of the 231-conditions, then the sufficiency.
  
\medskip

\noindent\emph{Proof of Necessity.}

\medskip

The monotonicity conditions are clear by Lemma \ref{inc_by_1}, and it is also clear that $S(P,F)$  starts and ends with the value 0.  If the values of $S(P,F)$ at two successive vertices were both $0$, then if one of these vertices were below (respectively, to the left of) the other, $F$ would have a row (respectively, a column) with no marker in it, contradicting the fact that $P$ is a full rook placement.  

We now prove the diagonal condition by induction on the number of squares in $F$.  For a board with one square, it is obvious that the 231-diagonal condition holds for the only possible placement.  Assume now that $P \in S_F(231)$ and the result holds for all boards with fewer squares than $F$.  

\medskip

\noindent{\tt{Case 1:}} $\mathcal{B}$ contains a marker.  

\medskip

 Let $V_0,\ldots V_{2n}$ be the sequence of vertices on the right/up border of $F$ starting at the top left corner of $F$, and let $B=V_k$.   Since $S(P,V_{k-1}) =  S(P,V_{k+1})$ by Lemma \ref{lemma_growthrule}, it will suffice to check the diagonal condition for all diagonal vertices not containing $V_{k+1}$.  To this end denote by $F^*$ and $P^*$ the board and placement obtained by deleting the row and column of $\mathcal{B}$ from $F$.  Now let $V^*_i$ be the vertex directly under $V_i$ for $0\leq i\leq k-1$ and the vertex directly to the left of $V_i$ for $k+2\leq i\leq 2n$.  Observe that the sequence 
$$V^*_0,\ldots,V^*_{k-1},V^*_{k+2},\ldots,V^*_{2n}$$ 
is precisely the sequence of vertices on the right/up border of $F^*$. Fix $i,j \notin \{k,k+1\}$.  It is clear that 
\begin{equation}
S(P,V_i) = S(P^*,V^*_i)
\end{equation}
and that
\begin{equation}
V_i, V_j \textrm{ are } F\textrm{-diagonal vertices} \text{ iff } V^*_i, V^*_j\textrm{ are }F^*\textrm{-diagonal vertices.}
\end{equation}  
By induction $S(P^*,F^*)$ satisfies the diagonal condition.  Therefore (1) and (2) directly imply that $S(P,F)$ also satisfies the diagonal condition.

\medskip

\noindent{\tt{Case 2:}} $\mathcal{B}$ does not contain a marker.  

\medskip

Note that in this case we must have $r\geq 2$, and consider the smaller board $F^* = F\setminus \mathcal{B}$.  By the induction hypothesis the pair $(F^*,S(P,F^*)))$ satisfies the diagonal condition.  So we only need to show that $S(P,A)\leq S(P,C)$.   Since $\mathcal{B}$ contains no marker, Lemma \ref{equality_of_labels} implies that $S(P,C) = S(P,B)$.  By monotonicity we must have $S(P,A)\leq S(P,B)$ which concludes this case. $\Box$

\medskip
 
\noindent\emph{Proof of Sufficiency.}

\medskip

We prove the sufficiency of the 231-conditions by again using induction on the number of squares in $F$.  Let $S$ be an $F$-sequence such that $(F,S)$ satisfies the 231-conditions, and let $a,b,c$ denote $S(A),S(B),S(C)$, respectively.

\medskip
 
\texttt{Case 1}: $b\neq a$ and $b\neq c$.  

\medskip

First note that in this case we must have $a+1=b=c+1$ by monotonicity of $S$.  Define $V_0,\ldots, V_{2n}$ as in the proof of necessity, with $B=V_k$.   Also define $F^*$,  $P^*$,  and vertices $V^*_i$ as in that proof, so that the sequence 
$$V^*_0,\ldots, V^*_{k-1},V^*_{k+2},\ldots,V^*_{2n}$$
is precisely the sequence of vertices on the right/up border of $F^*$.  

Now define $S^*(V_i^*)=S(V_i)$ for $i\notin \{k,k+1\}$.  We claim that $(F^*,S^*)$ satisfies the 231-conditions. Since $a=c$ it is clear that $(F^*,S^*)$ satisfies the monotonicity conditions. To see that $(F^*,S^*)$ satisfies the 0-conditions it suffices to show that $a\neq 0$ when $r\geq 2$. So suppose $r\geq 2$ and $a=0$.  Draw the diagonal $\ell$ extending NW from $A$ and let $V_1$ be the first vertex on the right/up border where $\ell$ passes outside of $F$.  Since $A$ is above the diagonal from upper left to lower right, there must be a vertex $V_2$ on the right/up border directly to the left of $V_1$.  Since $a=0,$ the diagonal and monotonicity conditions for $(F,S)$ yield $S(V_1)=0$ and $S(V_2)=0,$ contradicting the 0-conditions for $(F,S)$.

To verify the diagonal condition for $(F^*,S^*),$ note that:  $V_i,V_j$ are $F$-diagonal vertices if and only if $V_i^*,V_j^*$ are $F^*$-diagonal vertices.  Therefore since $(F,S)$ satisfies the diagonal condition so must $(F^*,S^*)$.

Since $(F^*,S^*)$ satisfies the 231-conditions, there exists, by the induction hypothesis, $P^*\in S_{F^*}(231)$ such that $S(P^*,F^*)=S^*$. Now restore the row and column we removed from $F$ and place a marker $X$ in square $\mathcal{B}$ to obtain a placement $P$ on $F$.  It is clear that $P\in S_F(231),$ because of the position of $X$. Lastly we show that $S(P,F)=S$.   Note that for $V_i\neq B$ or $C$
\begin{equation}
S(P,V_i)=S(P^*,V^*_i)=S^*(V^*_i)=S(V_i).
\end{equation}
Since $P$ has a marker in $\mathcal{B}$, we have $S(P,A)+1 = S(P,C)+1 = S(P,B)$. By (3) we know that $S(P,A) =a$.  Putting these together we have that $S(P,B)= a+1=b$ and $S(P,C)=a=c$.

\medskip
 
\texttt{Case 2}: $b=a$ or $b=c$.

\medskip

Note that in this case we cannot have $r=1$, because if $r=1$ then $c=0$, so by the diagonal condition for $(F,S)$, $a=0$ and thus $b=0$, violating the 0-conditions for $(F,S)$. So we can let $D$ be the vertex directly below $C$. Let $E$ be the vertex at the SW corner of $\mathcal{B}$, and let $d$ denote $S(D)$.  Denote by $F^*$ the Ferrers board $F\setminus \mathcal{B}$.  

Now consider the function $S^*$ defined by  

$$S^*(V) = \left\{\begin{array} {ll}
S(V)& \textrm { if }  V\neq E\\
\min{(a,d)} & \textrm { if }  V = E
\end{array}\right.$$
where $V$ is a vertex on the right/up border of $F^*$.
  
In order to apply the induction hypothesis to the smaller pair $(F^*, S^*)$ we need to know that $(F^*,S^*)$ satisfies the 231-conditions. Since $r\geq 2,$ we have $a\neq 0$ as in Case 1, and $a\leq c$. Since $(F,S)$ satisfies both the monotonicity and 0-conditions  it easily follows that $(F^*,S^*)$ satisfies these two conditions as well. 

So it only remains to show that $(F^*,S^*)$ satisfies the diagonal condition. Now for the diagonal extending SE from $E$ we have $ S^*(E)=\min(a,d) \leq d = S^*(D)$.  Next consider the diagonal extending NW from $E$ and let its right-most intersection point with the right/up border be $E_0$. (Note that $E_0$ exists since $r\geq 2$.)  Call the vertex to $E_0$'s immediate right $A_0$ and note that $A_0$ must be on the right/up border.  Our choice of $E_0$ implies that $A$ and $A_0$ are diagonal vertices.  Now if $\min(a,d) = a$ then by our definitions we have
$$S^*( E_0 )\leq S^*(A_0) = S(A_0)\leq S(A) = S^*(E).$$
If on the other hand $\min(a,d) = d$ then clearly $S^*(E_0)\leq S^*(E)$ since $E_0$ and $D$ are diagonal vertices in $F$.  Therefore $(F^*,S^*)$ satisfies the diagonal condition.  

Since the pair $(F^*,S^*)$ satisfies the 231-conditions then by the induction hypothesis there exists a 231-avoiding full rook placement $P$ on $F^*$ such that $S^* = S(P,F^*)$. We claim that $P$ is also a 231-avoiding full rook placement on $F$ such that $S=S(P,F)$.

To see that $S=S(P,F)$ let $V$ be any vertex on the right/up border of $F$. If $V\neq B$  then we have $S(P,V)=S^*(V)=S(V)$.
If $V=B$ then since $\mathcal{B}$ does not contain a marker we have $S(P,B)= \max(a, c) = b$ where the last equality holds because $\max{(a,c)}\leq b$ by the monotonicity of $S$ and $b\leq \max(a,c)$ since $b=a$ or $b=c$ in this case.

Lastly we need to show that $P$ is a 231-avoiding placement on $F$.  Assume it is not and let $XYZ$ be a 231-pattern in $F$.    Let marker $Y$ be in square $\mathcal{B}_1$ and $Z$ be in square $\mathcal{B}_2$.  Note that square $\mathcal{B}_1$ must be in row $r$, along with square $\mathcal{B}$. Likewise, note that  $\mathcal{B}_2$ must be in the right-hand column. Since $P$ in $F^*$ has no 231-patterns then all the markers in the columns strictly between $\mathcal{B}_1$ and $\mathcal{B}$  must be above row $r$.  For if not then some marker $W$ is either in a row below $X$'s row in which case $XYW$ is a 231-pattern in $F^*$, or $W$ is in a row between $X$'s row and $Y$'s row resulting in the 231-pattern $XWZ$.  So if $\bar{A}$ and $\bar{E}$ denote the vertices in the NE and SE corners of $\mathcal{B}_1$ respectively then it follows that, letting $e=\min{(a,d)}$,
$$ \begin{array}{ccc}
S(P,\bar{A})=S(P,A)=a&\textrm{ and }& S(P,\bar{E})=S(P,E)=e.\end{array}$$
If we could show that $a=e,$ it would follow from Lemma \ref{lemma_growthrule} that $\mathcal{B}_1$ could not contain a marker.  But this would be a contradiction as $\mathcal{B}_1$ contains the marker $Y,$ and we would be done.  To show that $a=e$,   note  that $Z$ cannot be in row $r-1$, because $XYZ$ is a 231-pattern.  Since row $r-1$ must contain  a marker, Lemma \ref{equality_of_labels} implies that $c=d$ and therefore $e =\min(a,d) = a$ since $a\leq c$ by the diagonal condition.  $\Box$

\medskip

\centerline{\textbf{References}}

\medskip

\footnotesize
\noindent1. J. Backelin, J. West, G. Xin, Wilf-equivalence for
singleton classes, \emph{Adv. Appl. Math.} \textbf{38} (2007), 133--148.\\[5pt]
2. S. Fomin, Generalized Robinson-Schensted-Knuth correspondence, \emph{Zapiski Nauchn. Sem. LOMI} \textbf{155} (1986), 156--175.\\[5pt]
3. V. Jel\'inek, Dyck paths and pattern--avoiding matchings, \emph{European J. Comb.} \textbf{28} (2007), 202--213.\\[5pt]
4. Z. Stankova, J. West, A new class of Wilf-equivalent permutations, \emph{J. Alg. Comb.} \textbf{15} (2002), 271--290.\\[5pt]
\end{document}